\documentclass[11pt]{article}
\usepackage{graphicx} 
\usepackage{amsthm}
\usepackage{amssymb}
\usepackage{amsmath}
\usepackage{svg}
\usepackage[a4paper]{geometry}
\newtheorem{teo}{Theorem}[section]
\newtheorem{cor}{Corollary}[section]
\newtheorem{prop}{Proposition}[section]

\numberwithin{equation}{section}

\newtheorem{lem}{Lemma}[section]
\theoremstyle{definition}
\newtheorem{defn}{Definition}[section]

\theoremstyle{remark}
\newtheorem{remark}{Remark}[section]
\newtheorem{ex}{Example}[section]

\newcommand{\Z}{\mathbb{Z}}
\newcommand{\Q}{\mathbb{Q}}
\newcommand{\R}{\mathbb{R}}

\newcommand{\G}{\mathcal{G}}

\newcommand{\sgn}{\operatorname{sgn}}

\newcommand{\pff}{\operatorname{pff}}
\newcommand{\p}{\mathcal{P}}
\newcommand{\inv}{^{-1}}

\newcommand{\Bo}{A_{\overrightarrow{G}}}

\title{A determinantal formula for cluster variables in cluster algebras from surfaces }
\author{Javier De Loera Chávez}
\date{}

\begin{document}

\maketitle

\begin{abstract}
    For cluster algebras of surface type, Musiker, Schiffler and Williams gave a formula for cluster variables in terms of perfect matchings of snake graphs. Building on this, we provide a simple determinantal formula for cluster variables via the weighted biadjacency matrix of the associated snake graphs, thus circumventing the enumeration of their perfect matchings.  
\end{abstract}

\section{Introduction}

Cluster algebras are a broad class of commutative algebras introduced by Fomin and Zelevinsky in \cite{FZ1} as a combinatorial framework for the study of total positivity and canonical bases. They are recursively defined from an initial set of data by a combinatorial process called mutation. In \cite{FST}, Fomin, Shapiro, and Thurston introduced and studied a class of cluster algebras whose combinatorics are governed by flips of tagged triangulations in surfaces with marked points. In \cite{MSW}, Musiker, Schiffler and Williams gave cluster expansion formulae for the cluster variables of these cluster algebras in terms of the perfect matchings of certain planar graphs called \textit{snake graphs}. These graphs were studied further in \cite{CS}, where a remarkable connection to continued fractions was developed. 

In this note, we use a classical idea of matching theory to give a reformulation of Musiker's, Schiffler's, and Williams' cluster expansion formula into a determinantal formula for the cluster variables associated to plain arcs (see Proposition \ref{mio2}). This gives a simple way to avoid the enumeration of the perfect matchings of the corresponding snake graph. The formula works by using an algebraic device to count all of the perfect matchings of a graph. The formulae for cluster variables associated to tagged arcs in \cite{MSW} are given in terms of more restrictive matchings, simplified in \cite{wilson}, were the expansion formulae involve \emph{good matchings} of so-called \emph{loop graphs}. Thus the scope of our strategy is limited to the case of plain arcs. 

We begin by briefly recalling the definition of surface cluster algebras following \cite{FST}. Let $\Sigma$ be a Riemann surface with boundary $\partial \Sigma$, and let $\mathbb{M} \subset \Sigma $ be a non-empty finite set of \textit{marked points} such that $\mathbb{M} \cap C \neq \emptyset$ for each connected component $C$ of $\partial \Sigma$. The pair $(\Sigma, \mathbb{M})$ is a \textit{bordered surface}, the marked points are called $ \partial \Sigma$ \textit{ideal points}, and those in the interior of $\Sigma$ \textit{punctures}.
A \textit{boundary arc} of $\Sigma$ is a curve with endpoints in $M$ that is contained in $\delta \Sigma$, and an \textit{arc} in $\Sigma$ is a curve $\gamma$ with endpoints in $\mathbb{M}$ such that its relative interior is disjoint from $M$, it has no self-intersections, and it is not homotopic to a boundary arc or to a point. All such arcs are considered up to isotopy. 
Two arcs are said to be \textit{compatible} if they do not intersect in their relative interiors, up to homotopy. A maximal collection $T$ of distinct pairwise compatible arcs in $\Sigma$ along with all the boundary arcs of $\Sigma$ is an \textit{ideal triangulation} of $(\Sigma, \mathbb{M})$. All ideal triangulations of $(\Sigma, \mathbb{M})$ have $6g+3b+3p+c-6$ arcs. It may happen that not all triangles of an ideal triangulation $T$ have three sides: when a loop $l$ encircles a puncture $p$, there is exactly one arc $r$ joining the base of $l$ to $p$ that is compatible with $l$. The shape formed by $l$ and $r$ is called a \textit{self-folded triangle}, where $r$ is the self-folded side.  

Pick an initial ideal triangulation $T_0$ with arcs $\tau_1, \ldots, \tau_n$. For each triangle $\Delta$ in $T_0$, let $B^\Delta = (b_{i,j}^{\Delta})_{1\leq i  \leq n, 1\leq j\leq n}$, where
\[
b_{i,j}^\Delta = \begin{cases}
        1 & \text{if $t_j$  follows $t_i$ in clockwise order,}\\
        -1 & \text{if $t_j$  follows $t_i$ in counterclockwise order,}\\
        0 & \text{otherwise}.
\end{cases}
\]
The \textit{signed adjacency matrix} $B=B(T_0)$ is given by $B = \sum_\Delta B^\Delta$, where the sum is taken over all triangles of $T_0$ that are not self-folded. 

Consider the \emph{tropical semifield} $\mathbb{P}:=\operatorname{Trop}(y_1,\ldots, y_n)$ defined as the multiplicative abelian group freely generated by $y_1, \ldots, y_n$ with the \emph{auxiliary sum} 
\[
    \prod_{i=1}^n y_i^{a_i} \oplus \prod_{i=1}^n y_i^{b_i} := \prod_{i=1}^n y_i^{\min (a_i,b_i)}\,.
\]

Take a tuple of $\textbf{x}_0 = (x_1, \ldots, x_n)$ of algebraically independent generators of the field of rational functions $\Q\mathbb{P}(u_1, \ldots, u_n)$. Such a tuple is called a \emph {cluster}.
Let $\mathbf{y}_0\in \mathbb{P}^n$. The triple $(\textbf{x}_0, \textbf{y}_0,B)$ is called the \textit{initial seed}. For each $k\in \{1,\ldots, n\}$ define the \textit{mutation} in direction k $\mu_k(\textbf{x}_0,\textbf{y}_0, B_0) := (\mathbf{x}', \mathbf{y}', B')$ by the following rule: 

\begin{itemize}
    \item The matrix $B' = (b'_{ij})$ is defined by\[
b'_{i,j} = \begin{cases}
        -b_{i,j} & \text{if } k\in\{i,j\} \,,\\
        b_{i,j} + \frac{|b_{i,k}|b_{k,j} + b_{i,k}|b_{k,j}|}{2} & \text{otherwise}\,.
    \end{cases}
\]
    \item The cluster $\mathbf{x}' = (x_1',\ldots, x_n')$ is given by $x_{i}'=x_{i}$ if $i\neq k$, and
    \[
    x_{k}' = 
        \frac{1}{(y_k\oplus1)x_{k}}\left(y_k\prod_{B_{k,j}>0} x_{j}^{B_{k,j}} + \prod_{B_{k,j}<0} x_{j}^{-B_{k,j}}\right)\,.\]

    \item The coefficient tuple $\mathbf{y}'$ is given by
    \[
    y_j'=\begin{cases}
        y_j\inv & i = k\,,\\
        y_j y_k^{\max (0,b_{kj})}(y_k \oplus1)^{-b_{kj}} & \text{otherwise}\,.
    \end{cases}
    \]
\end{itemize}

The matrix $B' = \mu_k(B(T_0))$ is in fact the adjacency matrix of the triangulation obtained by flipping the edge $\tau_k$, and the set $\mu_k(\textbf{x}_0)$ is again a cluster.
Let $\mathcal{X}$ be the set of \textit{cluster variables}: this is, all variables contained in all possible clusters obtained by sequences of mutations starting from the initial seed $\textbf{x}_0$. The \textit{cluster algebra} with principal coefficients associated to the surface with marked points $(\Sigma, \mathbb{M})$ is the algebra $\mathcal{A} := \mathbb{Q}\mathbb{P}[\mathcal{X}]$. It is independent of the chosen triangulation $T_0$, so it really is associated to the surface $(\Sigma, \mathbb{M})$. The cluster algebra $\mathcal{A}$ has principal coefficients (at $t_0$) when $\mathbf{y}_0 = (y_1, \ldots, y_n)$, and this is denoted by $\mathcal{A}_\bullet(T_0)$. 

There is a bijection between clusters of $\mathcal{A}$ and \textit{tagged} triangulations of $(\Sigma, \mathbb{M})$, as well as between cluster variables and tagged arcs in $\Sigma$. With this in mind, we will also write $y_{\tau_i}$ for $y_i$. The tagging of arcs introduced in \cite{FST} allows to fully model the mutation of all cluster variables associated to any initial tagged triangulation, since the self-folded side of a self-folded triangle cannot be flipped in the usual way. 

\section{Perfect matchings}

We will consider only finite and loopless graphs. Let $\G$ be a graph with vertex set $V=\{v_1, \ldots, v_n\}$ and edge set $E$. Denote by $E(i,j)$ the set of edges between $v_i$ and $v_j$. Write $s(G)$ for the underlying simple graph of $G$, and $m(i,j)=|E(i,j)|$. Of course, if $G$ is simple then $s(G) = G$ and $m(i,j) \in \{0,1\}$ for all $1\leq i,j\leq n$. If $\overrightarrow{G}$ is an orientation of $G$, then $E(i,j)$ will also denote the set of arrows with tail $v_i$ and head $v_j$.

A \textit{perfect matching} of $G$ is a subset $M\subseteq E$ such that each $v\in V$ is an endpoint of exactly one member of $M$. 
Thus, $|M| = |V|/2$, so a graph that admits a perfect has an even number of vertices. Call such a graph \textit{even}. Denote by $\operatorname{Match}(G)$ the set of perfect matchings of $G$, and let $\Phi(G) := |\operatorname{Match}(G)|.$ We will use a simple but useful device to keep track of all perfect matchings of $G$, the so-called \textit{method of variables}, as in \cite{LP}. 

\begin{defn}(Weighted adjacency matrix)\label{wam}
    Let $G$ be a graph with vertex set $V=\{v_1, \ldots, v_n\}$. For each edge $e$ of $s(G)$, consider the formal variable $z_e$, and let $a_{i,j} := m(i,j) z_e$.
    The matrix $A_G (z) := (a_{i,j})$ is called the \textit{weighted adjacency matrix} of $G$. Furthermore, if $\overrightarrow{G}$ is an orientation of $G$ with no oriented 2-cycles, write
    \[
    a'_{i,j} := \epsilon(i,j) m(i,j) z_e \,, \quad \text{where} \quad \epsilon(i,j) :=\begin{cases}
    1 & E(i,j) \neq \emptyset \,,\\
    -1 & E(j,i) \neq \emptyset \,, \\
    0 & \text{ otherwise}\,.\
    \end{cases} 
    \]
    $A_{\overrightarrow{G}}(z) := (a'_{i,j})$ is the \textit{weighted skew-symmetric adjacency matrix} of $\overrightarrow{G}$. 
\end{defn}

\begin{remark} 
The matrices of definition \ref{wam} are defined up to a relabeling of the vertex set of $G$, and thus a simultaneous permutation of rows and columns. By setting $z_e = 1$, we recover the usual adjacency matrices of $G$ and $\overrightarrow{G}$, so we denote them by $A_G(1)$ and $A_{\overrightarrow{G}}(1)$, respectively. 
\end{remark}

\begin{defn}
    Let $G$ be a graph and $E'$ be the edge set of $s(G)$. The \textit{matching generating function} of $G$ is the polynomial $\phi_G \in \Z[z_e \mid e\in E']$ given by
    \[
        \phi_{G} := \sum_{\{e_1, \ldots, e_{k}\} \in \operatorname{Match}(G)} z_{e_1}\cdots z_{e_{k}}  \, .
    \]
\end{defn}

\begin{defn} \label{pfaff}
    Let $C = (c_{i,j})$ be a skew-symmetric matrix of size $2n\geq 2$ over a commutative unit ring $A$. For each partition in pairs $P = \{\{i_1, j_1\}, \ldots, \{i_n, j_n\}\}$ of the set $\{1, 2, \ldots, 2n-1, 2n\}$, define 
    \[
        \sigma_P := \begin{pmatrix}
            1 & 2 & \ldots & 2n-1 & 2n \\
            i_1 & j_1 & \ldots & i_n & j_n
           \end{pmatrix} \, ,\quad c_P :=  \operatorname{sgn}(\sigma_P) c_{i_1,j_1}\cdots c_{i_n, j_n} \,,
    \]
    where $\operatorname{sgn}(\sigma_P)$ denotes the sign of the permutation $\sigma$.
    The \textit{pfaffian} of $C$ is $\operatorname{pff}(C) := \sum_{P} c_P \, . $
\end{defn}

\begin{remark}\label{independent}
    The term $c_P$ is independent of the order of labeling of each part in $P$. Indeed, switching the roles of $i_k$ and $j_k$ changes the sign of $\sigma_P$ by $-1$, but $c_{j_k,i_k} = -c_{i_k, j_k}$. Thus, $c_P$ depends only on the partition $P$. Note that such a partition corresponds to a perfect matching of $K_{2n}$, the complete graph in $2n$ vertices, and that $c_P \neq 0$ if and only if such a potential matching corresponds to an actual matching of $G$.
\end{remark} 
The pfaffian and determinant of a skew-symmetric matrix are closely related. The following lemma is due to Cayley \cite{C}. A proof can be found in \cite{A}, where it appears as Theorem 5.6.
\begin{lem}\label{cay}
    If $B$ is a skew-symmetric matrix of even size, $\det B = \operatorname{pff}(B)^2$.
\end{lem}

The next lemma gives a way to keep track of the perfect matchings of a graph $G$.

\begin{lem}
Let $G$ be an even graph, and let $\overrightarrow{G}$ be an arbitrary orientation of $G$. The sum of the absolute value of the coefficients of $\operatorname{pff}(A_{\overrightarrow{G}}(z))$ is $\Phi(G)$. 
\end{lem} 

\begin{proof}
    Take the vertex set $V$ of $G$ to be $\{1,\ldots, n\}$, and consider the simple graph $s(G)$ with edge set $E'$. Let $k=n/2$. Each partition $P = \{\{i_l, j_l\} \mid 1\leq l \leq k\}$ as in definition \ref{pfaff} corresponds to a possible set of independent edges in the simple graph $s(G)$. For $C=\Bo(z)$, the associated term $c_P$ is
    \[
        c_P = \begin{cases}
            0 & \text{if }\{i_s, j_s\}\notin E' \text{ for some } 1\leq s\leq k \, ,\\ 
            \prod_{l=1}^k \epsilon(i_l,j_l) m(i_l,j_l)z_{\{i_l, j_l\}} & \text{otherwise}.
        \end{cases}
    \]
    Thus, $c_P \neq 0$ if and only if $P$ corresponds to a perfect matching $M$ in $s(G)$, as $P$ is a partition. Note that the coefficient $\prod_{l=1}^k m(i_l,j_l)$ 
    is the number of ways of getting the matching $M$ in $s(G)$ from a matching in $G$.
    Since different matchings of $s(G)$ differ by at least one edge, terms corresponding to distinct matchings cannot cancel each other. This completes the proof.
\end{proof}
As observed by Tutte in \cite{T}, the previous argument and Lemma \ref{cay} imply that a graph $G$ admits a perfect matching if and only if there is some orientation $\overrightarrow{G}$ for which $\det (\Bo(z)) \neq 0$. Call the polynomial $\operatorname{pff}(\Bo(z))$ the \textit{pfaffian} of $\overrightarrow{G}$, and denote it by $\pff(\overrightarrow{G})$. Note that this polynomial is defined only up to a sign that depends on the chosen labeling of the vertices of $G$.

\begin{cor}
    Let $\overrightarrow{G}$ be an oriented graph. Then $
    |\operatorname{pff}(\Bo(1))| \leq \Phi(G)$.
\end{cor}
An orientation $\overrightarrow{G}$ of $G$ for which the equality is attained in the previous corollary is a \textit{pfaffian} orientation. Likewise, a graph is \textit{pfaffian} if it admits a pfaffian orientation. For such an orientation, the pfaffian of $\overrightarrow{G}$ is a polynomial whose coefficients all have the same sign. Indeed, in such a case $\pff(\overrightarrow{G}) = \pm \phi_G$. We cite a useful characterization of pfaffian orientations. Call a cycle $C$ of a graph $G$ \textit{nice} if $G-V(C)$ admits a perfect matching. An oriented even cycle is \emph{oddly oriented} if it has an odd number of arrows that agree with any of its two opposite cyclic orientations.
\begin{teo}(\cite{LP}, Theorem 8.3.2)\label{car}
    Let $G$ be an even graph and $\overrightarrow{G}$ an orientation of $G$. Then $\overrightarrow{G}$ is pfaffian if and only if each even nice cycle of $G$ is oddly oriented.
\end{teo}
Furthermore, the following is classic. A proof can be found in \cite{LP}, Theorem 8.3.4.
\begin{teo}\label{kast}(Kasteleyn)
    Every even planar graph is pfaffian. 
\end{teo}

Let $\overrightarrow{G}$ be bipartite with bipartition $V=U\cup W$ such that $|U| = |W|$ (if $\overrightarrow{G}$ has a perfect matching, this is the case). The weighted skew symmetric adjacency matrix of $\overrightarrow{G}$ takes the form
\begin{equation}\label{bloc}
    \Bo(z) = \begin{pmatrix} 
        0 & B_{\overrightarrow{G}}(z) \\ 
        -B_{\overrightarrow{G}}(z)^T & 0 
    \end{pmatrix} \,,
\end{equation}
for any labeling $V=\{u_1, \ldots, u_k, w_1,\ldots, w_l\}$ such that $u_i \in U$ and $w_j \in W$ for $1\leq i\leq k$, $1\leq j \leq l$.
$B_{\overrightarrow{G}}(x)$ is called the \textit{weighted biadjacency matrix} of $\overrightarrow{G}$.  
The following lemma is given as an exercise in \cite{G} (though there seems to be a sign missing therein). We include a proof. 
\begin{lem}\label{cuad}
    Let $B$ be a square matrix of size $n$. If
    \[
    A = \begin{pmatrix} 
        0 & B\\ 
        -B^T & 0 
    \end{pmatrix} \,,
\]
then $\pff(A) = (-1)^{n(n-1)/2} \det B$.  
\end{lem}
\begin{proof}
    Let $V=\{1,\ldots, 2n\}$, $U=\{1, \ldots, n\}$ and $W=\{n+1, \ldots, 2n\}$. Fix some partition $P = \{\{i_1, j_1\}, \ldots, \{i_n, j_n\}\}$. By remark \ref{independent}, we can assume that $j_i = n+i$ for each $1\leq i\leq n$, so
    \[
    \sigma_P = \begin{pmatrix}
                1 & 2 & 3 & 4 & \ldots & 2n-1 & 2n \\
                i_1 & n+1 & i_2 & n+2 & \ldots & i_n & 2n
                \end{pmatrix} 
    \]
    Note that $i_1, \ldots, i_n \leq n < n+1 = j_1$, so the indices $i_l$ amount to $n-1$ inversions with respect to $j_1$. Similarly, $n-2$ of the $i_l$ give $n-2$ inversions with $j_2$, and so on. This gives a total of $n(n-1)/2$ inversions. Thus,
    \[
        \operatorname{sgn}(\sigma_P) = (-1)^{n(n-1)/2}\operatorname{sgn} \begin{pmatrix}
                1 & 2 & \ldots & n-1 & n & n+1 & \ldots & 2n-1 & 2n\\
                i_1 & i_2 & \ldots & i_{n-1} & i_n & n+1 & \ldots & 2n-1 &2n
                \end{pmatrix}\,.
\]
It follows that $\operatorname{sgn}(\sigma_P) = (-1)^{n(n-1)/2}\operatorname{sgn}(\sigma_P ')$, where
\[
                \sigma_P ' := \begin{pmatrix}
                1 & 2 & \ldots & n-1 & n \\
                i_1 & i_2 & \ldots & i_{n-1} & i_n 
                \end{pmatrix}\,.
\]
Thus, $b_P = (-1)^{n(n-1)/2}\operatorname{sgn}(\sigma_P ')b_{i_1,n+1}\cdots b_{i_n,2n}$. Summing over all $P$ we obtain the result.
\end{proof}


\section{Snake graphs and cluster algebras}

A \textit{tile} is a graph isomorphic to a square drawn in the plane. A \textit{snake graph} $\G$ is a plane graph obtained from a sequence of $d$ tiles $G_1,\ldots, G_d$ such that each successive tile is placed on top or to the right of the previous tile, as in figure \ref{sg}. That is, for $1\leq i\leq d-1$, tiles $G_i$ and $G_{i+1}$ share their right and left edge, respectively, or their top and bottom edge, respectively. The edges of $\G$ that lie on its unbounded face are called \textit{boundary edges}, and the rest of the edges are \textit{internal}. It is easily seen by induction that snake graphs are bipartite and always admit perfect matchings, and that a snake graph with $d$ tiles has $2d+2$ vertices and $3d+1$ edges.

Snake graphs were used in \cite{MSW} to provide manifestly positive expansions of the cluster variables in cluster algebras from surfaces. We briefly recall their construction. Fix a bordered surface $(\Sigma, \mathbb{M})$ and an initial seed $(\mathbf{x}_0, \mathbf{y}_0, B(T_0))$, where $T_0$ is an ideal triangulation and let $\mathcal{A}_\bullet(T_0)$ be the associated cluster algebra with principal coefficients at $0$. As discussed in the introduction, each tagged arc $\gamma$ in $\Sigma$ corresponds to a cluster variable $x_\gamma$, and if $\gamma \notin T_0$, then $x_\gamma$ is not part of the initial cluster $\textbf{x}_0$, so one must perform a series of mutations to find the expansion of $x_\gamma$ with respect to $\textbf{x}_0$. This can be avoided by virtue of the following.
\begin{teo}(\cite{MSW}, Theorem 4.9)\label{mus}
    Orient $\gamma \notin T_0$ so that it crosses the diagonals $\tau_{i_1}, \ldots, \tau_{i_d}$ of $T_0$ in order. Each $\tau_{i_j}$ is contained in exactly two triangles $\Delta_j$ and $\Delta_j'$. Let $G_j$ be the quadrilateral formed by the triangles $\Delta_j$ and $\Delta_j'$. Give each side $s$ of $G_j$ the weight $w(s) = x_k$ if it is the internal diagonal of $T_0$ associated to $x_k$, and give it the weight $w(s) =1$ if it is a boundary edge of $T_0$. There is a snake graph $\G_\gamma$ with tiles $G_1, \ldots, G_d$ such that
    \[
    x_\gamma = \frac{1}{\operatorname{cross}(\gamma, T_0)} \sum_{M \in \operatorname{Match}(\G_\gamma)}x(M)y(M)\,
    \]
    where $x(M)$ is the product of all weights of the edges occurring in the matching $M$, $y(M)$ is the specialized height monomial of $M$, and $\operatorname{cross}(\gamma, T_0) := x_{i_1}x_{i_2}\cdots x_{i_d}$.
\end{teo}
For complete details on how the tiles $G_j$ are glued together to obtain $\G_\gamma$, see \cite{MSW}. There are exactly two matchings of $\G_\gamma$ consisting of boundary edges, $M_-$ and $M_+$. If tile $G_1$ is drawn in the plane in such a way that its orientation agrees with its orientation on $\Sigma$, $M_-$ is the matching that includes the south edge of $G$, and it is called the \emph{minimal matching} of $\G_\gamma$. We recall the definition of the specialized height monomial from \cite{MSW2}. 

\begin{defn}\label{height}
    Let $T_0 = \{\tau_1, \ldots, \tau_n\}$ be an ideal triangulation of $(\Sigma, \mathbb{M})$ and $\gamma\notin T_0$ a plain arc. By lemma 4.7 of \cite{MSW}, the symmetric difference $M \ominus M_{-}$ encloses a union of tiles $\cup_{j\in J} G_{i_j}$. The \emph{specialized height monomial} of $M$ is given by 
    \[
        y(M) = \prod_{j\in J} y_{\tau_{i_j}}\,.
    \]
\end{defn}
That is, $y(M)$ is the product of the variables $y_i$ that correspond to the diagonals of the tiles enclosed by $M\ominus M_{-}$. Note that several tiles of $\G_{\gamma}$ may have the same diagonal.  
By theorem 5.3 of \cite{MSW2}, the set of perfect matchings of $\G_\gamma$ is a distributive lattice $\operatorname{Match(\G_\gamma)}$. In this lattice, a matching $M$ covers another matching $M'$ when $M$ can be obtained from $M'$ via a local moved called a \emph{face twist}: that is, there is a tile $G_i$ of $\G_\gamma$ with edges $e_1,e_2,e_3,e_4$ such that $M\ominus M' = \{e_1,e_2,e_3,e_4\}$ and $M=(M'\setminus \{e_1, e_2\})\cup\{e_3,e_4\}$. The perfect matching lattice of $\G_\gamma$ can be built up from sequences of face twists beginning at the minimal matching $M_{-}$. Thus, if the matching $M$ is obtained by applying the sequence of face twists at tiles $G_{j_1},\ldots, G_{j_k}$, $y(M) = y_{\tau_{j_1}}\cdots y_{\tau_{j_k}}$.

We now describe an alternative weight scheme for the edges of $\G_{\gamma}$ that recovers the specialized height monomial of each matching. 
\begin{figure}\label{sg}
    \centering
    \includegraphics[scale=0.75]{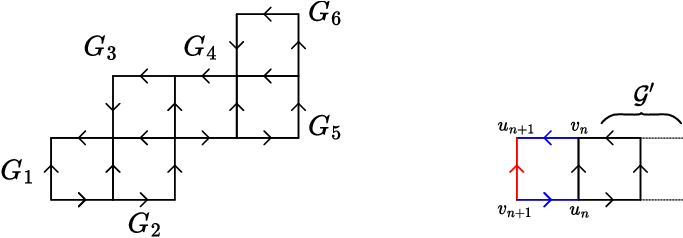}
    \caption{To the left, a snake graph $\G$ with the pfaffian orientation $\p(\G)$. To the right, the inductive step in the proof of Proposition \ref{mio}}
    \label{fig:enter-label}
\end{figure}

\begin{defn}(Principal weighting)
    Let $\G_\gamma$ be the snake graph on $n$ tiles associated to the arc $\gamma\notin T_0$ as in theorem \ref{mus}, and let $w$ be its corresponding weighting. Consider the minimal matching $M_{-}$. We define the weighting $\tilde{w}$ of the edges of $\G_\gamma$ as follows: for each $1\leq i\leq n$
    \begin{itemize}
        \item If it exists, pick a boundary edge of $e$ of $G_i$ such that $e\notin M_{-}$ and set $\tilde{w}(e) = y_{\tau_i}$.
        \item If all boundary edges of $G_i$ are in $M_{-}$, then the consecutive tiles $G_{i-1},G_i,G_{i+1}$ form a straight snake subgraph of $\G_\gamma$. Let $e,e'$ be boundary edges of $G_{i-1}$ and  $G_{i+1}$, respectively, and let $d,d'$ be the edges of $G_i$ that are not in $M_{-}$. Set $\tilde{w}(e')=\tilde{w}(d)=\tilde{w}(d')=\tilde{w}(e') =\sqrt{y_i}$.
        \item For all other edges $a$ of $\G_{\gamma}$, set $\tilde{w}(a) = 1$.
    \end{itemize}
    Let $w_\bullet$ be the weighting of $\G_\gamma$ defined by $w_\bullet(e) := w(e)\tilde{w}(e)$ for each edge $e$. We call $w_\bullet$ a \emph{principal weighting} of $\G_{\gamma}$.
    \end{defn}

\begin{prop}\label{principal}
    Let $\G_\gamma$ be the snake graph of the arc $\gamma \notin T_0$. For each $M\in \operatorname{Match}(\G_\gamma)$, let $w_\bullet(M) := \prod_{e\in M}w_\bullet(e)$. Then
    \[
        x_\gamma = \sum_{M\in \operatorname{Match}(\G_\gamma)} w_\bullet(M) \,.
    \]
\end{prop}
\begin{proof}
    By virtue of theorem \ref{mus}, we just need to show that $w_\bullet(M) = x(M)y(M)$. By definition, $x(M)$ divides $w_\bullet(M)$ and there is no power of $x_i$ that divides $w_\bullet(M)/x(M)$, for any $i$. Thus, it is enough to show that each variable $y_j$ appears with equal multiplicity in both $w_\bullet(M)$ and $y(M)$. Now, $y_j$ appears once in $y(M)$ for each tile $G_{i}$ with diagonal $\tau_j$ that is enclosed by the edges $M\ominus M_{-}$. If there is a boundary edge $e$ of $G_i$ in $M$, then $e\notin M_{-}$, so $M$ includes all boundary edges of $G_i$ (one if $G_i$ is a corner tile, two otherwise), and therefore $y_j \mid w_\bullet(M)$. If there is no boundary edge of $G_i$ in $M$, then both are in $M_{-}$, and $G_i$ is in such a position as in the third tile of figure \ref{ejem}. Let $e, e'$ be the edges in $G_i\cap M_{-}$, with $e$ adjacent to $a_1, a_2$ and $e'$ adjacent to $a_1', a_2'$. Since $G_i$ is enclosed by $M\ominus M_{-}$, $M$ must include one element $e_1$ of $\{a_1, a_2\}$ and one element $e_2$ of $\{a_1', a_2'\}$. In all cases, $w_\bullet(e_1)=\sqrt{y_j}=w_\bullet(e_2)$, so $y_j \mid w_{\bullet}(M)$. Reciprocally, any matching $M$ that includes the edge $a_1$ must also include either $a_1'$ or $a_2'$, so $M\ominus M_-$ encloses $G_i$ and $y_j$ appears in both $y(M)$ and $w_\bullet(M)$.
\end{proof}

Now, by Theorem \ref{kast}, snake graphs are pfaffian, and Theorem \ref{car} immediately suggests a pfaffian orientation:

\begin{defn}
    Let $\G$ be a snake graph with tiles $G_1, \ldots, G_d$. Define the orientation $\p(\G)$ as follows:
    \begin{enumerate}
        \item Orient the boundary edges of $\G$ cyclically and counterclockwise.
        \item Orient every internal vertical edge from bottom to top, and every internal horizontal edge from right to left. 
        \item Reverse the orientation of the left edge of $G_1$.
    \end{enumerate}
\end{defn}
The following is readily seen on account of Theorem \ref{car}, but we give a direct proof.
\begin{prop}\label{mio}
    For every snake graph $\G$, $\p(\G)$ is a pfaffian orientation.
\end{prop}

\begin{proof}
    We proceed by induction in $d$, the number of tiles. When $d=1$, say $\G$ has vertices $u_1, w_1, u_2, w_2$ listed in cyclic order. The weighted biadjacency matrix of $\p(\G)$ is 
    \[
        \begin{pmatrix}
            z_{u_1 w_1} & z_{u_2 w_1} \\
            -z_{u_1 w_2} & z_{u_2 w_2}
        \end{pmatrix}
    \]
    so that $\pff(\p(\G)) = z_{u_1 w_1}z_{u_2 w_2} + z_{u_2 w_1}z_{u_1 w_2}$. 
    Take $\G$ with tiles $G_1, \ldots, G_{d+1}$, $d\geq 1$ and give it the orientation $\p(\G)$. Let $\G'$ be the snake graph with tiles $G_2, \ldots, G_{d+1}$ obtained by deleting the two leftmost vertices of $\G$. Observe that the induced orientation on $\G'$ is $\p(\G')$, so by induction the coefficients of the pfaffian of $\G'$ all have the same sign $\epsilon \in \{1,-1\}$. By circling the boundary of $\G$ counterclockwise, we obtain a bipartition $\{u_1, \ldots, u_n\} \cup \{w_1, \ldots, w_n\}$ of its vertices. Relabel $V(\G)$ so that $G_1$ has vertices $w_{n+1}, u_n, w_n, u_{n+1}$ as in figure \ref{sg}, and write $a_{i,j}$ and $a_{i,j}'$ for the entries of $A_{\p(\G)}(z)$ and $A_{\p(\G')}(z)$, respectively. Note that $A_{\p(\G)}(x)$ is obtained from $A_{\p(\G')}(z)$ by adding a row and a column  corresponding to vertices $u_{n+1}$ and $w_{n+1}$.
     
    Any perfect matching $M$ of $\G$ either contains the edge $u_{n+1}w_{n+1}$ or the edges $u_{n+1}w_{n}$ and $u_n w_{n+1}$. 
    In the first case, $M$ restricts to a perfect matching of $\G'$, and its associated term in $\pff(\p(\G))$ is 
    \begin{align*}
            a_M &=\sgn \begin{pmatrix}
            1 & 2 & \ldots & 2n-1 & 2n & 2n+1 & 2n+2 \\
            i_1 & j_1 & \ldots &i_n & j_n & 2n+1 & 2n+2
        \end{pmatrix} a_{i_1, j_1} \ldots a_{i_n, j_n} (-z_{u_{n+1} w_{n+1}}) \\
        &= \sgn \begin{pmatrix}
            1 & 2 & \ldots & 2n-1 & 2n   \\
            i_1 & j_1 & \ldots &i_n & j_n  
        \end{pmatrix} a_{i_1, j_1} \ldots a_{i_n, j_n} (-z_{u_{n+1} w_{n+1}})\\
        &= -\epsilon z_{u_{i_1}w_{j_1}}\cdots z_{u_{i_n}w_{j_n}}z_{u_{n+1} w_{n+1}}
    \end{align*}
    In the second case, $M$ contains edges $u_{n+1}w_{n}$ and $u_n w_{n+1}$. Such matchings are in bijective correspondence to matchings $M'$ of $\G'$ containing $u_n w_n$. The associated term in $\pff(\p(\G))$ is
    \begin{align*}
         a_M &=\sgn \begin{pmatrix}
            1 & 2 & \ldots & 2n-1 & 2n & 2n+1 & 2n+2 \\
            i_1 & j_1 & \ldots &2n-1 & 2n+2 & 2n+1 & 2n
        \end{pmatrix} a_{i_1, j_1} \ldots a_{i_{n-1}j_{n-1}} (-z_{u_{n} w_{n+1}})(-z_{u_{n+1} w_n})\,.
    \end{align*}
    The corresponding term $a_{M'}$ is
    \[
         \sgn \begin{pmatrix}
            1 & 2 & \ldots & 2n-1 & 2n  \\
            i_1 & j_1 & \ldots &2n-1 & 2n
        \end{pmatrix} a_{i_1, j_1} \ldots a_{i_{n-1}j_{n-1}}z_{u_n, w_n}.
    \]
    Note $a_M$ can be obtained from $a_{M'}$ by substituting $z_{u_n, w_n}$ with $-z_{u_{n} w_{n+1}}z_{u_{n+1} w_n}$, for their associated permutations have opposite signs. Thus, $a_M = -\epsilon z_{u_{i_1}, w_{j_1}}\cdots z_{u_{i_{n-1}},w_{j_{n-1}}}z_{u_{n} w_{n+1}}z_{u_{n+1} w_n}$, so every term of $\pff(\p(\G))$ has coefficient $-\epsilon$.
\end{proof}
\begin{figure}\label{pol}
    \centering
    \includegraphics[scale=0.8]{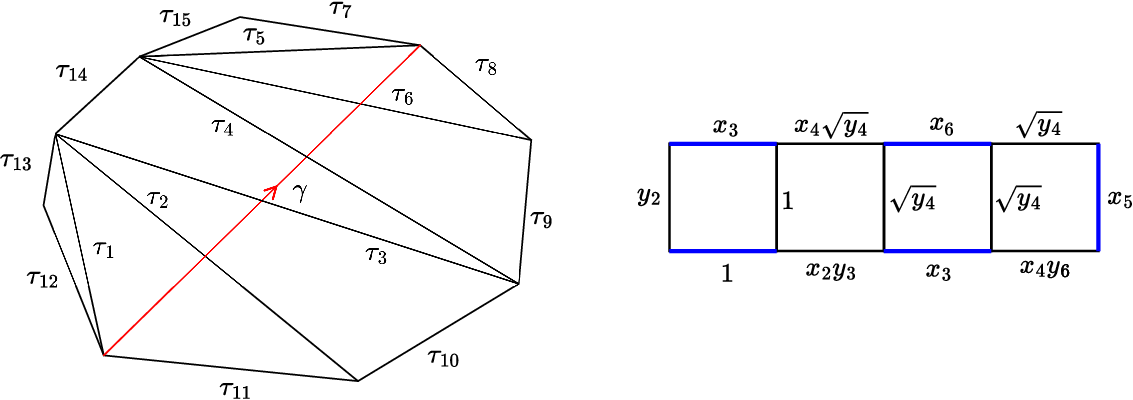}
    \caption{Example in type $A$. To the right, a principal weighting of $\G_\gamma$ with $M_{-}$ depicted in blue.}
    \label{ejem}
\end{figure}
\begin{remark}\label{pos}
    As in the previous proof, one can label the vertices cyclically by circling the boundary of $\G$ in the clockwise direction, starting from the leftmost bottom vertex $v_1$. Call this the \textit{cyclical labeling} of $\G$. With this labeling, the perfect matching of $\p(\G)$ given by the boundary edges $v_1 v_2, v_3 v_4, \ldots, v_{2n-1} v_{2n}$ has positive sign in $\pff(\G)$, as it is associated to the identity permutation. Thus, every coefficient of $\pff(\p(\G))$ is positive with respect to this labeling.
\end{remark}

\begin{remark}\label{pos2}
    The adjacency matrix of $\G$ with respect to the cyclic labeling is not of the form \ref{bloc}. To fix this, we relabel the vertices of $\G$ via the permutation $\sigma$ given by $\sigma(i) = (i+1)/2$ for odd $i$ and $\sigma(j) = i/2+n$ for even $j$, which maps the odd numbered vertices in the cyclic labeling of $\G$ onto the set $\{1,\ldots, n\}$, and the even numbered vertices onto the set $\{n+1, \ldots, 2n\}$. Moreover, $\sgn(\sigma) = (-1)^{n(n-1)/2}$, as can be seen by counting the negative slope segments with endpoints in $\{(i, \sigma(i)) \in \R^2 \mid i\in \{1,\ldots, 2n\}\}$.
\end{remark}

\begin{defn}
Let $\G$ be a snake graph. Let $B_{\p(\G)}(z)$ denote the weighted biadjacency matrix of $\G$ obtained from the relabeling in remark \ref{pos2}. That is, when walking the boundary cycle of $\G$ in the counterclockwise sense starting from the southwest vertex of the first tile of $\G$, we get the sequence of labels $1,
n/2+1, 2, n/2+2, \ldots, n/2, n$. 
\end{defn}
 
Note that all non zero entries of $B_{\p(\G)}(z)$ below the  have negative sign below the diagonal, and positive sign above it. In view of lemma \ref{cuad} and remarks \ref{pos} and \ref{pos2}, we have the following:

\begin{cor}\label{roc}
    Let $\G$ be a snake graph. Then $\det B_{\p(\G)}(z) = \phi_{\G}$
\end{cor}

\begin{remark}\label{otras}
    There is a shorter proof of proposition \ref{mio} that uses the lattice structure of $\operatorname{Match}(\G)$. To wit, let the matching $M$ be covered by $M'$, and write $\sigma_M$, $\sigma_{M'}$ for their associated permutations, with respect to some labeling of $V(\G)$. If $M'$ is obtained by twisting the face $f$ with the vertices $u_i, w_j, u_{k}, w_{l}$, then $\sigma'=(i \  l)\sigma$, but since $f$ is oddly oriented by theorem \ref{car}, $b_{i,j}b_{l,k} = - b_{i,l}b_{j,k} $, so the terms corresponding to $M$ and $M'$ have the same sign in $\pff \G$. Thus, we may more generally take any pfaffian orientation of $\G$ and label its vertices in such a way that for some given matching $M$ (say, the minimal matching $M_{-}$), the permutation $\sigma_M$ is the identity.
\end{remark}
With remark \ref{otras} in mind, we give the following definition. 
\begin{defn}
    Let $\G_\gamma$ be the snake graph associated to the plain arc $\gamma \notin T_0$. Fix any pfaffian orientation of $\G_\gamma$ and label its vertex set so that $\sigma_{M_{-}} = \operatorname{Id}$. Let $B(z)$ be the corresponding weighted biadjacency matrix of $\G_\gamma$. Denote by $B(w_\bullet)$ the matrix obtained by substituting the weight $w_\bullet(e)$ in each entry of $B(z)$. We will say $B(w_\bullet)$ is a \emph{principal} biadjacency matrix of $\G_\gamma$.
\end{defn}

We can now compactly state our determinantal formula, which follows from corollary \ref{roc}.

\begin{prop}\label{mio2}
    Let $\mathcal{A}(\Sigma, \mathbb{M})$ be a surface cluster algebra with principal coefficients at the initial seed $(T_0, \mathbf{x}_0, \mathbf{y}_0)$, where $T_0$ is an ideal triangulation of $\Sigma$. For every plain arc $\gamma \notin T_0$, let $\G_\gamma$ be its associated snake graph. Let $B(w_\bullet)$ be a principal biadjacency matrix of $\G_\gamma$. Then
    \[
    x_\gamma = \frac{\det B(w_\bullet)}{\operatorname{cross}(\gamma, T_0)} \,.
    \]
\end{prop}
\begin{ex}
Consider figure \ref{ejem}. Take the pfaffian orientation $\p(\G_\gamma)$ and the depicted principal weighting. The corresponding principal biadjacency matrix is
\[
B = \begin{pmatrix}
    1 & 0 & 0 &0& y_2 \\
    -x_2y_3 & x_3 & 0 & \sqrt{y_4} & 0\\
    0 & -x_4y_6 & x_5 & 0 & 0 \\
    0 & -\sqrt{y_4} & -\sqrt{y_4} & x_6 & 0 \\
    -1 & 0 & 0& -x_4 \sqrt{y_4} & x_3
\end{pmatrix}\, ,
\]
and $\det(B) = x_3x_5x_6y_2 + x_2 x_4^2 y_2 y_3 y_4 y_6 + x_2 x_4 x_5 y_2 y_3 y_4 + x_4 y_2 y_4 y_6 + x_5 y_2 y_4 + x_3^2x_5 x_6 + x_3 x_4 y_4 y_6 + x_3 x_5 y_4\,.$
\end{ex}

\section*{Acknowledegments}
The author received support from Daniel Labardini Fragoso's grant UNAM-PAPIIT-IN109824. This work was presented at the first meeting of the project ``Cluster algebras and related topics" financed by bando Galileo 2024- G24-215 and held at the department S.B.A.I. of Sapienza-Università di Roma in June 2024. The author thanks the participants, in particular Giovanni Cerulli Irelli, Pierre-Guy Plamondon and Ralf Schiffler for valuable comments, as well as Daniel Labardini-Fragoso for his kind support and encouragement. The author is grateful to Giovanni Cerulli Irelli and Salvatore Stella for their hospitality, and to Jon Wilson for his kind comments on a previous draft.

\end{document}